\documentclass[11pt]{article}
\usepackage{amsthm, amsmath, amssymb, amsfonts, url, booktabs, tikz, setspace, fancyhdr, bm}
\usepackage{hyperref}
\usepackage{geometry}
\usepackage{hyperref, enumerate}
\usepackage[shortlabels]{enumitem}
\usepackage[babel]{microtype}
\usepackage[english]{babel}
\usepackage[capitalise]{cleveref}
\usepackage{bbm,tkz-graph,subcaption}
\usepackage{csquotes}
\usepackage{mathrsfs}
\usepackage{mathabx}
\usetikzlibrary{patterns}
\usetikzlibrary{shapes}
\usepackage{bbm,wrapfig}

\usepackage[utf8]{inputenc} 
\usepackage[T1]{fontenc}
\usepackage{amsfonts}
\usepackage{amscd}
\usepackage{graphicx}
\usepackage{enumitem}
\usepackage{verbatim}
\usepackage{hyperref}
\usepackage{amsmath,caption}
\usepackage{url,pdfpages,xcolor,framed,color}
\usepackage{todonotes}
\usepackage{comment}

\newtheorem{thr}{Theorem}[section]
\newtheorem{q}[thr]{Question}

\newtheorem{prop}[thr]{Proposition}
\newtheorem{conj}[thr]{Conjecture}
\theoremstyle{definition}
\newtheorem{defi}[thr]{Definition}
\newtheorem*{defi*}{Definition}

\newtheorem{examp}[thr]{Example}

\def\E{\mathbb{E}}

\newcommand*{\abs}[1]{\left\lvert #1\right\rvert}

\def\a{\alpha}

\newcommand*{\myproofname}{Proof}

\tikzstyle{P} = [draw, circle, black, fill, inner sep = 0pt, minimum width = 3pt]
\tikzstyle{every loop} = []
\newcommand{\tikzHind}{
  \begin{tikzpicture}[baseline, yshift=1pt]
    \path[use as bounding box] (-.15,-.1) rectangle (.6,.35);
    \draw (0.5,0) node[P] {} -- (0,0) node[P] {} edge[-,in = 45, out = 135, loop] ();
  \end{tikzpicture}
}

\title{Counterexamples to conjectures on the occupancy fraction of graphs}

\date{}
\author{
Stijn Cambie \thanks{Department of Computer Science, KU Leuven Campus Kulak-Kortrijk, 8500 Kortrijk, Belgium. Supported by FWO grants with grant numbers 1225224N and 1222524N. \protect\href{mailto:stijn.cambie@hotmail.com}{\protect\nolinkurl{stijn.cambie@hotmail.com}} and \protect\href{jorik.jooken@kuleuven.be}{\protect\nolinkurl{jorik.jooken@kuleuven.be}}}
\and
Jorik Jooken \footnotemark[1] 
}

\begin{document}
\maketitle
\begin{abstract}
    The occupancy fraction of a graph is a (normalized) measure on the size of independent sets under the hard-core model, depending on a variable (fugacity) $\lambda.$ We present a criterion for finding the graph with minimum occupancy fraction among graphs with a fixed order, and disprove five conjectures on the extremes of the occupancy fraction and (normalized) independence polynomial for certain graph classes of regular graphs with a given girth. 
\end{abstract}

\section{Introduction}

We disprove five conjectures (\cite[conj.~2-5]{PP18} and~\cite[conj.]{CR18}), on the number of independent sets and the occupancy fraction of graphs, that have been open for five years\footnote{When considering arXiv, they are open for $7$ years}. This puts an end to failed proof searches.
For notation and definitions, the reader is referred to Subsection~\ref{subsec:not&def}.
The counterexamples indicate that finding the extremal graphs for the occupancy fraction among regular graphs with a given girth is a hard question in general, and the answer depends on the exact value of the fugacity.

The occupancy method was first introduced in Davies et al.~\cite{DJPR17}, where they extended the Kahn-Zhao theorem and its generalization.
The generalization of the Kahn-Zhao theorem (proven in steps by~\cite{Kahn01, GT04, Zhao10}) states that among $d$-regular graphs, the independence polynomial is maximized by $K_{d,d}$\footnote{after normalization, equivalently take the union of $K_{d,d}$s when $2d \mid n.$}. After taking the logarithm, this is equivalent with $\frac{\log P_G(\lambda)}{\abs{V(G)}}$ being maximized by $K_{d,d}$ (when restricting to $d$-regular graphs) for every $\lambda$ (the case where $\lambda=1$ gives the original Kahn-Zhao theorem). The latter result can be derived from integrating the occupancy fraction $\a_G(\lambda)= \frac{\lambda P'_G(\lambda)}{\abs{V(G)} P_G(\lambda)}$, which by~\cite{DJPR17} is maximized among $d$-regular graphs by $K_{d,d}$ as well, over $(0,1].$ 
We refer the interested reader to~\cite{DJPR17, Zhao17} for more information.

Perkins and Perarnau~\cite{PP18} studied the extrema of the occupancy fractions of $d$-regular graphs with a given girth condition.
The occupancy fraction $\alpha_G(\lambda)$ tends to $\frac{\alpha}{n}$ when $\lambda \to \infty.$

By Staton~\cite{Staton79}, it is known that for triangle-free cubic graphs the independence ratio $\frac{\alpha}{n}$ is at least $\frac 5{14}$.
Examples of graphs achieving equality include the generalized Petersen graph $P_{7,2}$ and the graph $G_{14}$ (see~\cref{fig:cubic_trianglefree}).
Their number of independent sets of a given size are given by
$\vec{i}(P_{7,2})=(1, 14, 70, 154, 147, 49)$ and $\vec{i}(G_{14})=(1, 14, 70, 154, 147, 48)$ respectively.
This comparison would immediately imply that~\cite[conj.]{CR18} and~\cite[Conj.~2]{PP18} are false.
It turns out that the natural adapted version is also incorrect and the two potential extreme examples, $P_{5,2}$ and $G_{14}$ do not result into the minimum over all possible values of $\lambda.$ The dodecahedral graph $DOD$ has a smaller occupancy fraction when $\lambda$ belongs to a certain interval.

In~\cref{sec:maxocfr}, we compute $\vec{i}(G)$ for the graphs we want to compare to disprove~\cite[Conj.~3, 4]{PP18}. From this, the polynomial $P_G(\lambda)$ and rational function $\a_G(\lambda)$ are immediately derived.
We end with an analytic comparison of functions.
In~\cref{sec:minocfr}, we prove a proposition that gives a sufficient condition for the inequality between the occupancy fractions of two graphs with the same order. The proof uses intermediate results that are proven in~\cref{sec:el_results_ocfr_sets} for sets in general.
By computing the critical graphs, we can determine the minimum occupancy fractions of triangle-free cubic and $4$-regular graphs with small order. By comparing with selected graphs with larger order,
we conclude that the extremal graphs are not (only) among the small ones (contrary to the intuition of~\cite[conj.~2-5]{PP18}).

All computer code of the verifications as well as an overview of the counterexamples can be found in the folder \url{https://github.com/JorikJooken/occupancyFraction}. We refer the interested reader to~\cref{sec:app1} for more details about the computer search that enabled us to find the counterexamples to the conjectures. 

\subsection{Notation and definitions.}\label{subsec:not&def}

In this paper, $G$ denotes a graph, and its order is denoted with $n$. We denote with $\a(G)$ the independence number of $G$, the number of vertices in a maximum independent set. We let $I(G)$ be the set of independent sets in $G$ and $\abs{I(G)}$ is the number of independent sets of the graph. Here the empty set is counted as well.

For a set $S,$ we consider the \emph{partition function} $P_S(\lambda)=\sum_{ s \in S} \lambda^s,$ which is a polynomial in the fugacity $\lambda$.
We define $\E_S(\lambda)=\frac{\lambda P'_S(\lambda)}{P_S(\lambda)}$.\footnote{$P'_S(\lambda)$ represents the derivative of $P_S(\lambda)$ with respect to $\lambda$}
This is the expected value when every element $s \in S$ is taken with a probability proportional to $\lambda^s$ and as such called the \emph{expected value} of the set $S$.
If the values of $S$ belong to $\{0,1,2, \ldots, k\},$ we call the list of non-negative integers representing the number of times (the multiplicity) each element appears in $S$, $(s_0, s_1, s_2,\ldots, s_k)$ the \emph{multiplicity-tuple} of $S$. Here $s_i$ denotes the multiplicity of $i$ in $S$.

\begin{defi}
    The number of independent sets with $k$ vertices in a graph $G$ is denoted with $i_k(G).$
\end{defi}

For a graph $G$ with independence number $k=\alpha(G)$, let $S$ be the set with multiplicity-tuple $\vec{i}(G)=(i_0(G), i_1(G), \ldots, i_{k}(G)).$
Then $P_G(\lambda)=P_S(\lambda)$ and $\a_G(\lambda)=\frac{1}{\abs{V(G)}}\E_S(\lambda)$ are the \emph{partition function} and \emph{occupancy fraction} of the graph $G$.

\section{Elementary results on expected value of sets}\label{sec:el_results_ocfr_sets}

One could expect that the expected value of a randomly drawn element according to a certain probability distribution, would decrease when the elements of the set decrease. It turns out this is not always the case in the hard-core model.

\begin{examp}
    Decreasing a value in a set, can increase the expectation.
    Let $S=\{0,2,4\}$ and $S'=\{0,1,4\}.$
    Then $\E_S(\lambda)=\frac{2\lambda^2+4\lambda^4}{1+\lambda^2+\lambda^4}<\frac{\lambda+4\lambda^4}{1+\lambda+\lambda^4}=\E_{S'}(\lambda)$ for small positive values of $\lambda$.
\end{examp}

While decreasing some elements does not necessarily lead to a decrease of the expected value, restricting the set to the smallest elements does.
When restricting to a subset of large elements, the expected value is larger.

\begin{prop}\label{prop:occupancy_subsets}
    Let $S$ be a set. Let $S=S_1 \cup S_2$ be a partition in some small and some large elements, i.e., where $\max S_1 \le \min S_2.$
    Then $\E_{S_1}(\lambda) \le \max S_1 \le \min S_2 \le \E_{S_2}(\lambda)$ for every $\lambda >0.$
    This implies that $\E_{S_1}(\lambda) \le\E_{S}(\lambda) \le \E_{S_2}(\lambda)$. 
\end{prop}

\begin{proof}
    The first inequality is trivial, as the expectation is between the minimum and maximum.
    The second inequality is then a corollary of the first, since
    $\E_{S}(\lambda)=\frac{\lambda P'_{S_1}(\lambda) + \lambda P'_{S_2}(\lambda)}{P_{S_1}(\lambda)+P_{S_2}(\lambda)}$ is a convex combination of $\E_{S_1}(\lambda)$ and $\E_{S_2}(\lambda)$.
\end{proof}

Using this observation, we can provide a criterion that is sufficient to compare the expected value.

\begin{prop}\label{prop:maj->min_setversion}
    Let $A$ and $B$ be the sets with multiplicity-tuples respectively $(a_0,a_1, \ldots, a_k)$ and $(b_0,b_1, \ldots, b_k)$, whose elements are positive integers.
    The condition $\frac{a_i}{a_{i-1}} \le \frac{b_i}{b_{i-1}}$ for every $1 \le i \le k$ ensures that 
    $\E_A(\lambda) \le \E_B(\lambda)$ for every $\lambda \ge 0.$
\end{prop}
\begin{proof}
    Due to the given condition, we have that 
    $$(b_0,b_1, \ldots, b_k) = c_0 (a_0,a_1, \ldots, a_k) + c_1(0,a_1,a_2, \ldots, a_k)+ \ldots + c_k( 0,\ldots,0,a_k),$$
    for some nonnegative reals $c_i.$
    We note that $\E_B(\lambda)$ is a convex combination of $\E_A(\lambda)$ and $\E_{A_{\ge i}}(\lambda)$ for $1 \le i \le k.$ Here $A_{\ge i} = \{a \mid a \in A, a \ge i\}$ is the subset of $A$ containing all elements of $A$ that are not smaller than $i.$
    By~\cref{prop:occupancy_subsets}, $\E_A(\lambda) \le \E_{A_{\ge i}}(\lambda)$ for every $i>0$ and the conclusion follows.
\end{proof}

\section{On maximum occupancy fraction}\label{sec:maxocfr}

In this section, we disprove~\cite[conj.~3 \& 4]{PP18}.

\begin{thr}
    There exists a cubic graph $G$ of girth at least $7$ for which $\a_G(\lambda) > \a_{H_{3,8}}(\lambda)$ for every $\lambda>17.$
\end{thr}

\begin{proof}
    Let $G=G_{38}$ be the graph in~\url{https://houseofgraphs.org/graphs/49521} ~\cite{HOG}.
    We can compute $\a_G(\lambda)$ and $\a_{H_{3,8}}(\lambda)$ explicitly from $\vec i(G)=(1,38,646,6498,43111,199120,658882,1583954,2777315$ $,3537622,3238356,2097330,947518,300924,72142,14802,2660,380,38,2)$ \\and $\vec i( H_{3,8})=(1, 30, 390, 2890, 13515, 41736, 86610, 120690, 111225, 66090, 24948, 6420, 1370, 240, 30, 2).$
    Now $\a_G(\lambda) - \a_{H_{3,8}}(\lambda)$ is a rational function, for which the sign depends on a polynomial of degree $23$.
    This polynomial has only one change of sign, and by Descartes' rule of signs, this implies that it has a single positive root.
    We conclude that $\a_G(\lambda) > \a_{H_{3,8}}(\lambda)$ for $\lambda$ larger than the positive root of this polynomial of degree $23.$
\end{proof}

Analogously, we get the following theorem.

\begin{thr}
    There exists a $4$-regular graph $G$ with girth at least $5$ such that $\a_G(\lambda) > \a_{H_{4,6}}(\lambda)$ for every $\lambda>37.$
\end{thr}

\begin{proof}
    Let $G=G_{32}$ be the graph in~\url{https://houseofgraphs.org/graphs/49999}.
We can compute $\a_G(\lambda)$ and $\a_{H_{4,6}}(\lambda)$ explicitly, and conclude that $\a_G(\lambda) > \a_{H_{4,6}}(\lambda)$ for $\lambda$ larger than the positive root of a polynomial of degree $17.$
\end{proof}

\section{On minimum occupancy fraction}\label{sec:minocfr}

In this section, we disprove~\cite[conj.~2 \& 5]{PP18} and~\cite[conj.]{CR18}.
For the study of the minimum occupancy fraction, we first prove the following proposition that compares graphs of the same order.

\begin{prop}\label{prop:maj->min}
    Let $G$ and $H$ be two graphs of order $n$.
    Suppose that $\a(G) \le \a(H)$ and
$$\frac{i_k(G)}{i_{k-1}(G)} \le \frac{i_k(H)}{i_{k-1}(H)} $$
    for every $1 \le k \le \a(G).$
    Then $\a_G(\lambda) \le \a_H(\lambda)$ for every $\lambda \ge 0.$
\end{prop}

\begin{proof}
    If $\a(G)=\a(H),$ this is immediate from~\cref{prop:maj->min_setversion}. 
    In the case that $\a(G)<\a(H)=\ell,$ we note that $(i_0(H), i_1(H), \ldots, i_{\ell}(H)) = (i_0(H), i_1(H), \ldots, i_{k}(H),0,\ldots, 0)+(0,0,\ldots,0,i_{k+1}(H), \ldots, i_{\ell}(H)).$
    Let the corresponding sets with these three multiplicity-tuples be $S,S_1,S_2.$
    Then~\cref{prop:maj->min_setversion} and~\cref{prop:occupancy_subsets} imply that $\a_G(\lambda) \le \frac 1n \E_{S_1}(\lambda) \le \frac 1n \E_{S}(\lambda)=\a_H(\lambda),$ as desired.
    \end{proof}

Using~\cref{prop:maj->min}, we can determine the extremal graphs for the occupancy fraction among cubic and $4$-regular triangle-free graphs of small order for every $\lambda >0.$

\begin{thr}\label{thr:PPconj2} Let $\lambda_1 \sim 1.21338$ be the positive root of $65x^6 + 110x^5 - 21x^4 - 144x^3 - 105x^2 - 30x - 3.$
    Let $\lambda_2 \sim 6.87002$ be the positive root of $-72x^8 + 180x^7 + 1639x^6 + 3158x^5 + 2777x^4 + 1276x^3 + 307x^2 + 34x + 1.$
    Then for every cubic triangle-free graphs with order bounded by $24$
    $$\a_G(\lambda) \ge \begin{cases}
        \a_{P_{5,2}}(\lambda) \mbox{ if }  \lambda \in (0, \lambda_1],\\
        \a_{DOD}(\lambda) \mbox{ if }  \lambda \in [ \lambda_1,\lambda_2],\\
        \a_{G_{14}}(\lambda) \mbox{ if } \lambda \in [ \lambda_2, +\infty).\\
    \end{cases}$$
\end{thr}

\begin{proof}
    For a fixed order, we call a graph $H$ critical if there is no graph $G$ such that~\cref{prop:maj->min} is true for these graphs.
    For every order, we compute the few critical graphs. At the end, we compare all of them and conclude.
\end{proof}

We note that with the same examples, we can disprove the conjecture in the end of~\cite{CR18}.

\begin{thr}
    Let $b_1 \sim 2.0927$ be the positive root of $x^3 - x^2 - 2x - 3/5=0$ and $b_2 \sim 17.264$ the positive root of a degree $51$ polynomial.
    The graph $DOD$ satisfies $P_{DOD}(\lambda)^{1/20}<P_{G_{14}}(\lambda)^{1/14}, P_{P_{5,2}}(\lambda)^{1/10}$ for $\lambda \in (b_1,b_2).$ 
\end{thr}

\begin{proof}
    As the polynomials have been computed from $\vec{i}(P_{5,2})=(1, 10, 30, 30, 5)$ , $\vec{i}(G_{14})=(1, 14, 70, 154, 147, 48)$ and $\vec{i}(DOD)=(1,20,160,660,1510,1912,1240,320,5)$,
    the conclusion is just derived from comparing polynomials.
    These comparisons have been done in\\ \url{https://github.com/JorikJooken/occupancyFraction/blob/main/CR_conj.pdf}.
\end{proof}

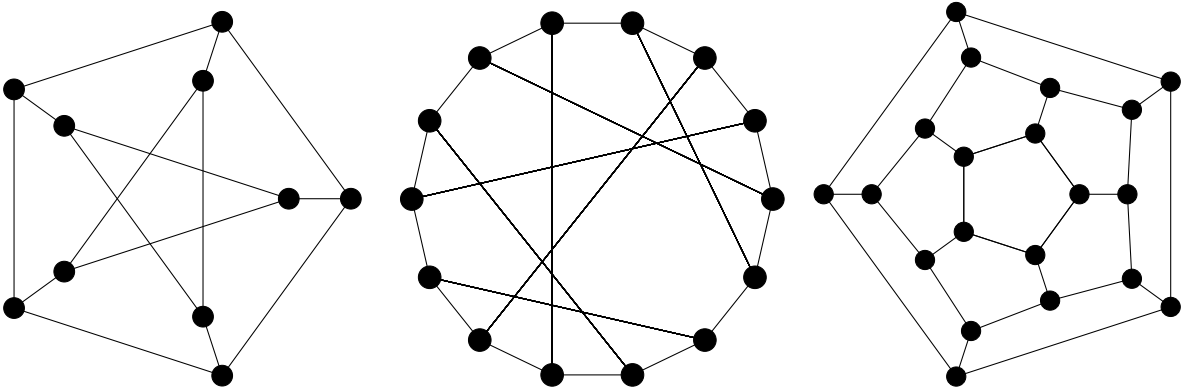
\begin{figure}[h]
\begin{center}
\begin{tikzpicture}[scale=0.55]
\foreach \x in {0,1,...,4}{
\draw[fill] (\x*360/5:4.5) circle (0.25);
\draw[fill] (\x*360/5:3) circle (0.25);
\draw (\x*360/5:3)--(\x*360/5:4.5);
\draw (\x*360/5+360/5:4.5)--(\x*360/5:4.5);
\draw (\x*360/5+720/5:3)--(\x*360/5:3);
}
\end{tikzpicture} \quad
\begin{tikzpicture}[scale=0.8]
\foreach \x in {0,1,...,13}{
\draw[fill] (\x*360/14:3) circle (0.1875);
\draw (\x*360/14+360/14:3)--(\x*360/14:3);
\draw (0:3)--(5*360/14:3);
\draw (-360/14:3)--(3*360/14:3);
\draw (360/14:3)--(7*360/14:3);
\draw (-2*360/14:3)--(8*360/14:3);
\draw (-4*360/14:3)--(4*360/14:3);
\draw (2*360/14:3)--(9*360/14:3);
\draw (-3*360/14:3)--(6*360/14:3);
}
\end{tikzpicture}\quad
\begin{tikzpicture}[scale=0.425]
\foreach \x in {0,1,...,4}
{
\draw[fill] (\x*360/5:2) circle (0.3);
\draw (\x*360/5+360/5:2)--(\x*360/5:2);
\draw[fill] (\x*360/5:3.5) circle (0.3);
\draw (\x*360/5:3.5)--(\x*360/5:2);
\draw (\x*360/5+360/5:2)--(\x*360/5:2);
\draw[fill] (\x*360/5+36:4.5) circle (0.3);
\draw (\x*360/5:3.5)--(\x*360/5+36:4.5);
\draw (\x*360/5:3.5)--(\x*360/5-36:4.5);
\draw[fill] (\x*360/5+36:6) circle (0.3);
\draw (\x*360/5+36:6)--(\x*360/5+36:4.5);
\draw (\x*360/5+36:6)--(\x*360/5+108:6);
}
\end{tikzpicture}
\end{center}
\caption{$P_{5,2}, G_{14}, DOD$}\label{fig:cubic_trianglefree} 
\end{figure}

\begin{thr}\label{thr:ag_CYC_ROB}
    Let $\lambda_3 \sim 1.77239$ be the positive root of $90x^7 + 819x^6 + 541x^5 - 1820x^4 - 2879x^3 - 1610x^2 - 401x - 38=0.$ Then for every $4$-regular triangle-free graph with order bounded by $19,$
    $$\a_G(\lambda) \ge \begin{cases}
        \a_{G_{\text{ROB}}}(\lambda) \mbox{ if }  \lambda \in (0, \lambda_3],\\
        \a_{\text{CYC}_{13}}(\lambda) \mbox{ if } \lambda \in [ \lambda_3, +\infty).\\
    \end{cases}$$
\end{thr}

\begin{proof}
    Similarly as in~\cref{thr:PPconj2}, we call a graph $H$ critical if there is no graph $G$ of the same order such that~\cref{prop:maj->min} is true for these graphs.
    For every order $8 \le n \le 19$ for which triangle-free $4$-regular graphs exist, we compute the critical graphs. We compute $\a_H(\lambda)$ for each of these critical graphs and by comparing all of them, we conclude.
    Details of the final comparisons can be found in~\url{https://github.com/JorikJooken/occupancyFraction/blob/main/PPconj5.pdf}.
\end{proof}

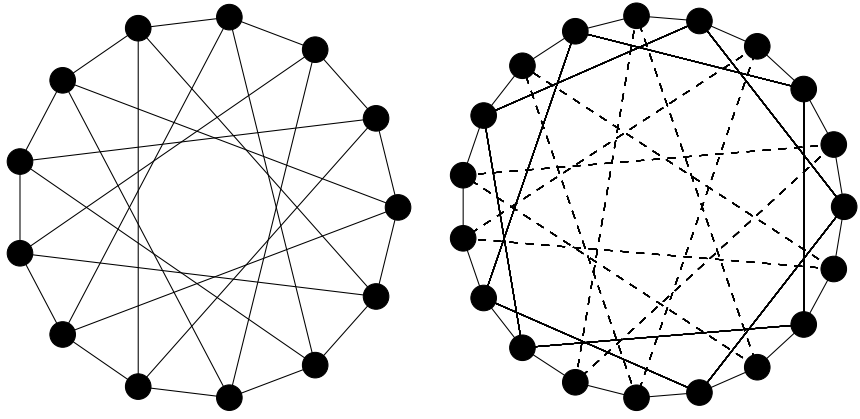
\begin{figure}[h!]
\begin{center}
\begin{tikzpicture}[scale=0.85]
\foreach \x in {0,1,...,12}{
\draw[fill] (\x*360/13:3) circle (0.2);
\draw (\x*360/13+360/13:3)--(\x*360/13:3);
\draw (\x*360/13+1800/13:3)--(\x*360/13:3);
}
\end{tikzpicture} \quad
\begin{tikzpicture}[scale=0.85]
\foreach \x in {0,1,...,18}{
\draw[fill] (\x*360/19:3) circle (0.2);
\draw (\x*360/19+360/19:3)--(\x*360/19:3);
\foreach \x in {-8,-4,0,4,8,2,17}{
\draw (\x*360/19+4*360/19:3)--(\x*360/19:3);
}
\foreach \x in {6,12}{
\draw (\x*360/19+5*360/19:3)--(\x*360/19:3);
}
\draw[dashed] (-360/19:3)--(10*360/19:3)--(3*360/19:3)--(14*360/19:3)--(7*360/19:3)--cycle;
\draw[dashed] (360/19:3)--(9*360/19:3)--(-3*360/19:3)--(5*360/19:3)--(13*360/19:3)--cycle;
}
\end{tikzpicture}
\end{center}
\caption{Cyclotomic-13 graph $\text{CYC}_{13}$ and
Robertson Graph $G_{\text{ROB}}$}\label{fig:Cyc13_Rob} 
\end{figure}

By focusing on a subset of the triangle-free $4$-regular graphs, we can find counterexamples to~\cite[conj.~5]{PP18}.
Let $G_{20}$ and $G_{22}$ be the graphs with graph$6$ presentation given by respectively\\
\texttt{S@?IC?g@S\_P?@aOWOS@ACSD@GGPCg?gB?} and \texttt{UIAC@OOA\_H@@?Qo?c\_?cH@O?OQD?GIC?OG\_`?KQ?}
; the two graphs are also presented in~\cref{fig:G20_G22_conj5}.

Let $\lambda_4 \sim 0.434965$ be the smallest positive root of
$8x^{10} - 724x^9 - 7040x^8 - 7612x^7 + 11321x^6 + 23576x^5 + 11895x^4 - 593x^3 - 2265x^2 - 679x =65$, $\lambda_5 \sim 1.23423$ be the positive root of $-24x^{11} - 952x^{10} - 11612x^9 - 25296x^8 - 9658x^7 + 25646x^6 + 36870x^5 + 22158x^4 + 7266x^3 + 1356x^2 + 137x + 6=0$ and
$\lambda_6 \sim 2.27938$ be the positive root of
$132x^7 + 252x^6 - 318x^5 - 1340x^4 - 1477x^3 - 747x^2 - 181x = 17.$

\begin{thr}
    For $\lambda \in (\lambda_4, \lambda_5),$ $\a_{G_{22}}(\lambda) <
        \min \{ \a_{G_{\text{ROB}}}(\lambda) ,
        \a_{\text{CYC}_{13}}(\lambda) \}.$
        
        For $\lambda \in (\lambda_5, \lambda_6), \a_{G_{20}}(\lambda) <
        \min \{ \a_{G_{\text{ROB}}}(\lambda) ,
        \a_{\text{CYC}_{13}}(\lambda) \}.$
\end{thr}

\begin{proof}
    We compute and compare the $4$ rational functions $ \a_{H}(\lambda)$ for $H \in \{\text{CYC}_{13}, G_{\text{ROB}}, G_{20}, G_{22}\}.$
\end{proof}

We finally note that neither $G_{\text{ROB}}$ nor $\text{CYC}_{13}$ minimizes $\frac{1}{\abs{V(G)}} \log \abs{I(G)}.$ Hereby $\text{CYC}_{13}$ cannot be extremal by integrating $\a_G(\lambda)$ over $(0,1]$ and~\cref{thr:ag_CYC_ROB}.
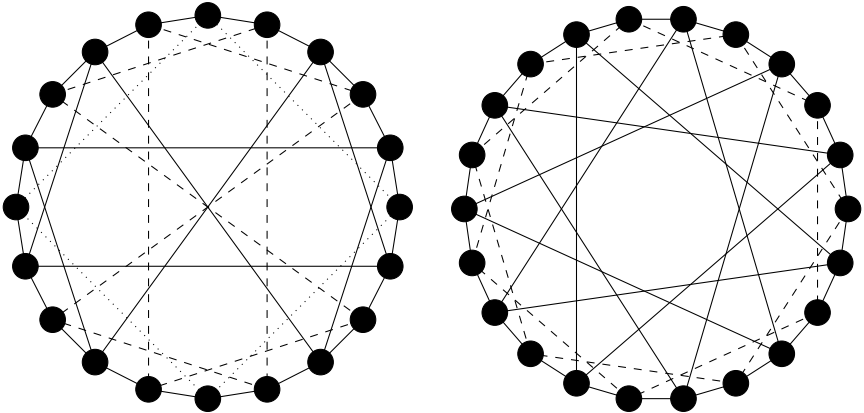
\begin{figure}[h!]
\begin{center}
\begin{tikzpicture}[scale=0.85]
\foreach \x in {0,1,...,19}{
\draw[fill] (\x*360/20:3) circle (0.2);
\draw (\x*360/20+360/20:3)--(\x*360/20:3);
}
\foreach \x in {0,5,10,15}{
\draw[dotted] (\x*360/20+5*360/20:3)--(\x*360/20:3);
}
\foreach \x in {-1,-3,7,9}{
\draw (\x*360/20+4*360/20:3)--(\x*360/20:3);
}
\foreach \x in {3,-3}{
\draw (\x*360/20+180:3)--(\x*360/20:3);
}
\foreach \x in {-1,1}{
\draw (\x*360/20:3)--(180-\x*360/20:3);
}
\foreach \x in {4,2,12,14}{
\draw[dashed] (\x*360/20+4*360/20:3)--(\x*360/20:3);
}
\foreach \x in {8,-8}{
\draw[dashed] (\x*360/20+180:3)--(\x*360/20:3);
}
\foreach \x in {4,6}{
\draw[dashed] (\x*360/20:3)--(-\x*360/20:3);
}
\end{tikzpicture}
\quad
\begin{tikzpicture}[scale=0.85]
\foreach \x in {0,1,...,21}{
\draw[fill] (\x*360/22:3) circle (0.2);
\draw (\x*360/22+360/22:3)--(\x*360/22:3);
}
\foreach \x in {0,1,2,...,10}{
\draw[dashed] (2*\x*360/22+4*360/22:3)--(2*\x*360/22:3);
\draw (2*\x*360/22+9*360/22:3)--(2*\x*360/22+360/22:3);
}
\end{tikzpicture}
 \end{center}
\caption{Graphs $G_{20}$ and $G_{22}$}\label{fig:G20_G22_conj5} 
  \end{figure}

\begin{thr}
    $$\frac{1}{22} \log \abs{I(G_{22})} < \frac{1}{19} \log \abs{I(G_{\text{ROB}})}$$
\end{thr}

\begin{proof}
    This follows from a direct calculation, knowing that $\abs{I(G_{22})}=6447$ and $\abs{I(G_{\text{ROB}})}=1950.$ 
\end{proof}

\section{Conclusion}\label{sec:conc}


By investigating the conjectures on the extremal values of occupancy fractions for certain classes of regular graphs with a minimum girth condition, we concluded that multiple such conjectures were incorrect when considering all possible fugacities. As such, the number of conjectures and questions that are likely to be true are narrowed down.
These conjectures are related to the study on the number of graph homomorphisms between graphs.
A homomorphism is a map from $V(G)$ to $V(H)$ that maps every edge of $G$ to an edge of $H$. The number of graph homorphisms from $G$ to $H$ (or the number of $H$-colorings on $G$) is denoted by $\hom(G, H)$.
Two of the main quantities generalized with the concept of graph homomorphisms, are the number of independent sets in a graph $G$, $\abs{I(G)} = \hom(G, \tikzHind)$, and the number of $q$-colorings of $G$, which is equal to $\hom(G, K_q)$.

One of the conjectures inspiring and summarizing questions on the number of graph homomorphisms was the one by Galvin,~\cite[conj.~1]{Galvin13} (and its precedent by Galvin and Tetali~\cite{GT04}), stating that for graphs $G$ and $H$, where $G$ is $d$-regular and has order $n$, the following inequality is true
$$\hom(G,H) \le \max \{ \hom(K_{d,d},H)^{n/2d} , \hom(K_{d+1},H)^{n/(d+1)}\}$$
The conjecture of Galvin was disproven by Sernau~\cite{Sernau17}. In his counterexample, $H$ is disconnected. 
Nevertheless, also when imposing that $G$ and $H$ are both connected, counterexamples can be found (see~\cref{sec:app2} for examples where $G$ and $H$ are simple, connected graphs).

We now highlight three interesting directions for future research.

\begin{itemize}
        \item No counterexample to~\cite[Ques.~2]{PP18} has been discovered.
\begin{q}(\cite[Ques.~2]{PP18})\label{q:PP_q2}
Is it true that for all graphs $H$ and all cubic graphs $G$ of girth at least $6$,
\[ \hom(G,H)^{1/|V(G)|} \le \hom(H_{3,6},H)^{1/14} \, ?\]
\end{q}
As an analog to higher girths of a result of~\cite{SSSZ20} (first conjectured by~\cite{CCPT17}),~\cref{q:PP_q2} having a positive answer seems very probable at this point.

    \item The remaining conjecture by Perkins and Perarnau (~\cite[Conj.~1]{PP18}) also seems to be true. This conjecture says that Moore graphs are extremal (attaining minimum or maximum depending on parity of girth) for the normalized independence polynomial, equivalently for $\frac{\log P_G(\lambda)}{\abs{V(G)}}$.
Hereby one of the approaches is to prove extremality of $\a_G(\lambda)$ for $ \lambda \in (0,1]$ for these Moore graphs. In particular, one can aim to do so for the Moore graphs $H_{4,6}, H_{3,8}$ which would be restrictions of~\cite[conj.~3,4]{PP18}. The analogous restriction of~\cite[conj.~2]{PP18} is proven in~\cite[Thm.~3]{PP18}.

    \item It is known by \cite[Cor.~A.2]{BSVV08}, and by an alternative proof by Csikv\'{a}ri in~\cite[Thm.~8.3]{Zhao17},
that $K_{d+1}$ minimizes the (normalized) number of $q$-colorings ($\hom(G, K_q)^{1/\abs{V(G)}}$) among $d$-regular graphs $G$ for any positive integer $q$. In the other direction, $K_{d,d}$ is the maximizer by~\cite[Thm.~1.4]{SSSZ20}.
As an analog of~\cite[Thm.~9.4]{Zhao17} (actually~\cite[Thm.~3,5]{PP18} uses girth $5$ instead of $C_4$-free for the maximizer), it would be natural to conjecture the same for $q$-colorings and as such extend~\cite[Cor.~A.2]{BSVV08} for higher girth conditions.
Again special care is needed for the adjusted conjecture and it turns out\footnote{E.g. $\hom(DOD,K_3)=7200<14400=\hom(P_{5,2},K_3)^2$} that one needs to impose that $q \ge 4$ for $P_{5,2}$ to be the minimizer among triangle-free cubic graphs.

\begin{conj} \label{conj:ind-girth}
(a) Provided that $q\ge 4,$ among $3$-regular triangle-free graphs $G$, the quantity $\hom(G, K_q)^{1/\abs{V(G)}}$ is minimized when $G$ is the Petersen graph.

(b) Among $3$-regular graphs $G$ without cycles of length $4$, the quantity $\hom(G, K_q)^{1/\abs{V(G)}}$ is maximized when $G$ is the Heawood graph.
\end{conj}
\end{itemize}

\section*{Acknowledgement}

The authors thank Guillem Perarnau and Will Perkins for updating us on the state and ongoing research on the conjectures in their paper, and suggestions for improvements of our paper.
We also thank P\'{e}ter Csikv\'{a}ri for informing us about~\cite{BSVV08} and sharing ideas which lead to~\cref{sec:app2}.

\paragraph{Open access statement.} For the purpose of open access,
a CC BY public copyright license is applied
to any Author Accepted Manuscript (AAM)
arising from this submission.

\bibliographystyle{abbrv}
\bibliography{ref}

\section*{Appendix}\label{sec: appendix}

\appendix

\section{Details about computer search}\label{sec:app1}

We used a computer to determine the independence polynomial of a graph (and the occupancy fraction) with two independent algorithms. The first algorithm uses the built-in functions in Sage, whereas the second algorithm was implemented in C++.  The latter is a simple, yet effective, backtracking algorithm that recursively enumerates all independent sets of a graph and derives the independence polynomial from this. The algorithm maintains an initially empty set and adds vertices recursively to this set by considering the vertices in an arbitrary order and branching into two possibilities: including the current vertex (if none of its neighbors are already in the set) or excluding the current vertex.

We used the graph generator GENREG~\cite{Meringer99} to exhaustively generate all connected $k$-regular graphs on $n$ vertices with girth at least $g$. We calculate the occupancy fraction for $\lambda \in \{0.1, 0.2, \ldots, 100.0\}$ (or keep track of the critical graphs using~\cref{prop:maj->min_setversion} for minimizing the occupancy fraction). This approach allowed us to find counterexamples for \cite[conj.~2]{PP18} and \cite[conj.]{CR18}. As the number of graphs on $n$ vertices grows rapidly with increasing $n,$ this approach became unfeasible for disproving the other conjectures (\cite[conj.~3-5]{PP18}). Therefore, we restricted the search space further by focusing on bipartite graphs (for maximizing the occupancy fraction) and graphs for which the automorphism group has at most two group orbits (see~\cite{HR20} for more details on their enumeration). The counterexamples that we found have between 14 and 38 vertices.

\section{Connected counterexample to Galvin's conjecture}\label{sec:app2}

We now present a connected counterexample for Galvin's conjecture when $d=3$.  

We will apply a graph product, called \emph{tensor product} (or \emph{categorical product}).\footnote{See also~\cite[Sec.~5]{Zhao17}.}
The tensor product $G \times H$ of two graphs $G$ and $H$ is the graph with vertex set $V(G) \times V(H)$, where $(u,v)$ and $(u',v') \in V(G) \times V(H)$ are adjacent in $G \times H$ if $uu' \in E(G)$ and $vv' \in E(H)$. 
Recursively, one can define $G^n$ as $G \times G^{n-1}$ for every $n \ge 2$.
The important property of the tensor product, making it very useful, is
\begin{equation}\label{eq:id_tensorproduct_hom}
    \hom(G, H_1 \times H_2) = \hom(G, H_1) \hom(G, H_2)
\end{equation}

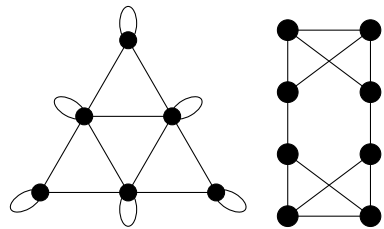
\begin{wrapfigure}{r}{7cm}
\begin{center}
\begin{tikzpicture}[scale=0.45]

\foreach \x in {0,1,2}{

\draw[fill] (120*\x+30:1.5) circle (0.25);
\draw[fill] (120*\x-30:3) circle (0.25);
\draw[rotate=120*\x+30] (0:2) ellipse (0.5cm and 0.25cm);
\draw[rotate=120*\x-30] (0:3.5) ellipse (0.5cm and 0.25cm);
\draw (120*\x+30:1.5) --(120*\x-90:1.5) ;
\draw (120*\x+90:3) --(120*\x-30:3) ;
}

\end{tikzpicture}\quad
\begin{tikzpicture}[scale=0.55]

\foreach \x in {0,1.5,3,4.5}{
    \foreach \y in {0,2}{

    \draw[fill] (\y,\x) circle (0.25);
    }
}
\draw (2,0) --(0,1.5) ;
\draw (0,0) --(2,1.5) ;
\draw (0,3)--(2,4.5) ;
\draw (2,3)--(0,4.5) ;

\draw (0,0) --(2,0)--(2,4.5)--(0,4.5)--(0,0) ;
\end{tikzpicture}
\end{center}
\caption{The graphs $H_0$ (left) and $G$ (right).}\label{fig:H1AndG1} 
\end{wrapfigure}

Let $H_0$ and $G$ be the (looped) complement of a net graph and the $K_4^-$-necklace, shown in~\cref{fig:H1AndG1} on the left and right, respectively. One can verify that $\hom(G,H_0)=58734$, $\hom(K_{3,3},H_0)=3732$, $\hom(G,K_3)=24$, $\hom(K_{3,3},K_3)=42$ and $\hom(K_4,K_3)=0$.

Now define the simple graph $H := H_0^{a} \times K_3$ for any $a \ge 216$. Using the identity~\eqref{eq:id_tensorproduct_hom}, we obtain:
\begin{align*}
    \hom(G,H) &> \max \{ \hom(K_{3,3},H)^{4/3} , \hom(K_{4},H)^{2}\}\\& = \hom(K_{3,3},H)^{4/3}.
\end{align*}

\end{document}